\documentclass[12pt]{article}

\usepackage{amsmath}
\usepackage{amssymb}
\usepackage{amsthm}
\usepackage{graphicx}
\usepackage{natbib}
\usepackage{color}
\usepackage{verbatim}
\usepackage{url}
\usepackage{mathtools}
\usepackage{underoverlap}
\usepackage{lineno}
\usepackage{bm}
\usepackage[title]{appendix}
\usepackage{xcolor}
\usepackage{soul}
\soulregister\citep7
\soulregister\ref7
\soulregister\suc7

\newtheorem{theorem}{Theorem}
\newtheorem{lemma}[theorem]{Lemma}
\newtheorem{corollary}[theorem]{Corollary}
\newtheorem{remark}{Remark}

\newcommand{\W}{r}
\newcommand{\WW}{R}
\newcommand{\s}{s}
\newcommand{\w}{w}
\newcommand{\bx}{\mathbf{x}}
\newcommand{\bz}{\mathbf{z}}
\newcommand{\suc}{successions}
\newcommand{\fun}{h}

\newcommand{\id}{h}
\newcommand{\Q}{P}

\newcommand{\A}{\mathbf{A}}

\newcommand{\M}{\mathbf{M}}
\newcommand{\U}{\mathbf{U}}
\newcommand{\V}{\mathbf{V}}

\newcommand{\bb}{\mathbf{b}}
\newcommand{\bc}{\mathbf{c}}
\newcommand{\be}{\mathbf{e}}
\newcommand{\by}{\mathbf{y}}
\newcommand{\bn}{\mathbf{n}}
\newcommand{\bone}{\mathbf{1_k}}

\newcommand{\G}{\mathcal{G}}
\newcommand{\HH}{\mathcal{H}}
\newcommand{\var}{\sigma}
\newcommand{\pd}{\partial}

\newcommand{\multinom}[2]{\genfrac{[}{]}{0pt}{}{#1}{#2}}

\begin{document}

\title{Distributions of \suc{} of arbitrary multisets}

\author{Yong Kong\\
School of Public Health\\
Yale University\\
300 Cedar Street, New Haven, CT 06520, USA\\
email: \texttt{yong.kong@yale.edu} }

\date{}
\maketitle
\newpage

\begin{abstract}
  By using the matrix formulation of the two-step approach to
  distributions of patterns in random sequences,
  recurrence and explicit formulas for the generating functions
  of \suc{} in random permutations of arbitrary multisets are derived.
  Explicit formulas for the mean and variance are also obtained.
\end{abstract}

%
%
%
%

%
%

\vskip 3mm

\newpage

\section{Introduction} \label{S:intro}

For a sequence $\sigma = ( a_1, a_2, \dots, a_n )$ where 
$a_i \in \{1, 2, \dots, k\}$,
we call $(a_{i}, a_{i+1})$ an \emph{increasing succession}
if $a_{i+1} - a_{i} = 1$.
If integer $j$ appears $n_j$ times in $\sigma$, 
we call $\sigma$ of \emph{specification} $\bn=[n_1, n_2, \dots, n_k]$,
with $\sum_{j=1}^k n_j = n$.
When $\bn=[1,1, \dots, 1] = [\bone]$, the set of all sequences $\{ \sigma_i \}$
is the usual symmetric permutation group.
As an example, when $k=3$, the specification $\bn=[2, 3, 2]$ specifies
two $1$'s, three $2$'s, and two $3$'s.
The following sequence with specification $\bn=[2, 3, 2]$ 
\[
  \UOLoverline{2 \enspace 3} \enspace 1 \enspace 3 \enspace \UOLoverline{1 \enspace 2} \enspace 2
  \]
has  two \suc{} $(2, 3)$ and $(1, 2)$, as indicated by the overhead bars.

The study of the number of random sequences with certain patterns,
while itself an interesting mathematical problem,
has found many applications in various fields~\citep{Balakrishnan2002,fu2003}.
With the development and omnipresence of computers, the study
  of patterns in random permutations has become invaluable in
  computer algorithm analysis~\citep{Knuth1998c}.
Many researchers have investigated various succession-related distributions
in the past.
Most earlier studies are on permutations with specification
$\bn=[1,1, \dots, 1] = [\bone]$
\citep{Whitworth1959, kaplansky1944, Riordan1945, Riordan1965, Abramson1966, Abramson1967, roselle1968, dwass1973, tanny1976, Reilly1980,johnson2001}.
The numbers of permutations of $[\bone]$ with no successions is given by
the following sequence for $k=1, 2, 3, \dots$
\begin{equation} \label{E:no-succ}
 1, 1, 3, 11, 53, 309, 2119, 16687, 148329, 1468457, \cdots
\end{equation}
(see A000255 of the On-Line Encyclopedia of Integer Sequences
\citep{Sloane2017}
and many references cited there).
This sequence is closely related to the
number of permutations of $n$ elements with no fixed points
(subfactorial or rencontres numbers, or derangements)
\cite[Chapter 4]{Whitworth1959}:
\[
 0, 1, 2, 9, 44, 265, 1854, 14833, 133496, 1334961, \cdots
\]
(See A000166  of the On-Line Encyclopedia of Integer Sequences
\citep{Sloane2017}).
%
%
Mysteriously, sequence~\eqref{E:no-succ} of permutations of $[\bone]$ with no successions
appears in Euler's transformation of the divergent series
$1 - 1! + 2! - 3! + 4! \cdots$
\cite[p. 29]{Hardy1949}.

The more difficult problem, the distribution of \suc{} in random permutations of
arbitrary multisets, however, has not  received much attention, and as Johnson stated,
``has remained unsolved for the last century''~\citep{Johnson2002}.
The successful attack of the problem~\citep{Johnson2002} came by using the versatile
finite Markov chain embedding technique
\citep{Fu1994,fu2003,Balakrishnan2002,Koutras1995}.
\citet{fu1995}
introduced the method to examine the distribution
of increasing \suc{} in $\bn=[\bone]$ and
later the method was applied to the multisets problem~\citep{Johnson2002}.

Recently we used a two-step approach to solve various problems in run and
pattern related problems
~\citep{Kong2006,Kong2016iwap,Kong2018,Kong2019}.
In this paper we use the matrix formulation of the
two-step approach to obtain the distribution of
increasing \suc{} with arbitrary multisets.
The main results are the recurrence and explicit formulas for
the generating function $\G_k$ of the whole system
in terms of the generating function $g_i$ of individual integers ($i=1, \dots, k$).
The explicit formulas for the mean and variance of the number of \suc{}
with arbitrary multisets are also derived.
As we have mentioned before~\citep{Kong2016iwap,Kong2019},
the advantage of expressing $\G_k$ in terms of $g_i$ is that,
by using appropriate $g_i$'s,
different kinds of distributions, with or without restrictions
on the number of individual integers ($i=1, \dots, k$),
can be readily handled
by straightforward specializations of the
generating functions $g_i$'s,
and various joint distributions can be obtained easily.
Compared with previous work~\citep{Johnson2002},
  this paper gives an explicit and concise formula for the generating function $\G_k$
  (Theorem~\ref{Th:explicit} and Corollary~\ref{C:G}),
  which covers all specifications for a given $k$,
  and a recurrence formula that makes it possible to
  calculate efficiently the number of sequences with exactly $\W$ \suc{}
  for a given specification $\bn$
  (Theorem~\ref{Th:P_recur}).
  In addition, explicit formulas are derived for the mean and variance
  for arbitrary specifications (Theorems~\ref{Th:mean} and \ref{Th:e12_var}).
  For a general discussion of the method upon which the results of this paper are based and
  the finite Markov chain embedding method, the readers may refer to the review~\citep{Kong2016iwap}.

The paper is organized as the follows.
First, in Section~\ref{S:Gk}, 
we obtain in Theorem~\ref{Th:rec}
a recurrence relation for the generating function $\G_k$ for \suc{} in
random permutations of arbitrary multisets,
and in Theorem~\ref{Th:explicit} the explicit formula of $\G_k$
is obtained in terms of $g_i$ of the individual integers.
Theorem~\ref{Th:P_recur} gives a recurrence relation
of \suc{} with respect to the numbers of a particular integer.
In Section~\ref{S:mean_var},
formulas for the mean and variance of \suc{} in random permutations
of arbitrary multisets are obtained.
In addition,
in Appendix~\ref{A:DM}, a formula  is given for the determinant of
$M_k$, the matrix used in the proof of Theorem~\ref{Th:rec}.
In Appendix~\ref{A:k2k3} explicit formulas of the number of \suc{}
in terms of specification $\bn$ are given for special cases
of $k=2$ and $k=3$.

In the following, for a given specification $\bn=[n_1, n_2, \dots, n_k]$,
the sequences are generated by permutations of integers from $1$ to $k$, 
and the length of the sequences is denoted by $n=\sum_{i=1}^k n_i$.
Let $\W$ be the number of \suc{} in a sequence,
and $\s_k(\W; \bn)$ be the number of sequences with
exactly $\W$ \suc{} for a given specification $\bn$.
The generating function of $\s_k(\W; \bn)$ for a particular specification $\bn$ is
denoted by
\[
  P_k(w, \bn) = \sum_\W  \s_k(\W; \bn) w^{\W} .
\]

Let $g_i(x_i)$, with $1 \le i \le k$, be the generating function of the
$i$th integer with indeterminate $x_i$ to track its number,
and $\G_{k}$ be the generating function for the whole system: 
\[
 \G_{k} = \G_{k} [ g_1(x_1), g_2(x_2), \dots,  g_k(x_k) ]
 = \sum_{\bn, \W} \s_k(\W; \bn) w^{\W} \bx^\bn
 = \sum_{\bn} P_k(w, \bn) \bx^\bn,
\]
where
$\bx = [x_1, x_2, \dots, x_k]$ and $\bx^\bn = x_1^{n_1} \cdots x_k^{n_k}$.
We use the notation $[x^n] f(x)$ to denote the coefficient of
$x^n$ in the series expansion of $f(x)$.
The $m$th falling factorial powers are defined as
$
 x^{(m)} = x(x-1)\dots(x-m+1)
$ for $ m > 0 $,
 with $x^{(0)} = 1$.
 Throughout the paper indeterminate $\w$ is used in generating functions
 to mark the increasing \suc.

\section{Recursive and explicit formulas of $\G_k$} \label{S:Gk}

First we derive recursive relations that
the generating function $\G_k$ and $\HH_k$ (an auxiliary function, defined below)
satisfy.
\begin{theorem}[Recurrence relation for $\G_{k}$] \label{Th:rec}
  The generating function $\G_{k}$ and
  the auxiliary function $\HH_k$ satisfy the following
  system of recursive relations:
\begin{equation}  \label{E:G_H_recur}
\begin{aligned}
\G_{k} &= \G_{k-1} + \dfrac{\displaystyle g_k  (1 + \G_{k-1}) (1 + \HH_{k-1})}
  {\displaystyle 1 - g_k \HH_{k-1}} ,\\
\HH_{k} &= \G_{k-1} + \dfrac{\displaystyle g_k  (\w + \G_{k-1}) (1 + \HH_{k-1})}
  {\displaystyle 1 - g_k \HH_{k-1}}  . 
\end{aligned}
\end{equation}  
with $\G_{1}=g_1$ and $\HH_{1}=\w g_1$. 
\end{theorem}

\begin{proof}
Based on the matrix formulation of the two-step approach to the
  distributions of runs
  \citep{Kong2016iwap,Kong2019},
  $\G_{k}$ is given by
  
\begin{equation} \label{E:matrix}
  \G_k = \be_k \M_k^{-1} \by_k,
\end{equation}
where $\be_k = [\underbrace{1, 1, \cdots, 1}_k]$,
$\by_k = [g_1, g_2, \cdots, g_k]^T$,
and
\begin{equation} \label{E:M}
 \M_k =   
  \begin{bmatrix}
           1 & -g_1 \w  & -g_1   & \cdots & -g_1 \\
 -g_2  &  1           & -g_2 \w  & \cdots & -g_2  \\
      \vdots & \vdots       & \vdots       & \vdots & \vdots      \\
 -g_{k-1}  & \cdots & -g_{k-1}   & 1        & -g_{k-1} \w \\
 -g_k     & \cdots & -g_k      & -g_k  & 1 
   \end{bmatrix}.
\end{equation}
We define $\HH_k$ as
\begin{equation} \label{E:H}
  \HH_k = \be_k \M_k^{-1} \bz_k,
\end{equation}
where $\bz_k = [g_1, g_2, \cdots, g_{k-1}, \w g_k]^T $.

The matrix $\M_k$ can be written in a block form as
\[
\M_k = \begin{bmatrix}
  \M_{k-1}   & \bb \\
  \bc          & 1
  \end{bmatrix}
\]
where $\bb = [-g_1, -g_2, \cdots, -g_{k-1} \w]^T = -\bz_{k-1}$,
and
$\bc = [-g_k, -g_k, \cdots, -g_k ] = -g_k \mathbf{e}_{k-1}$ .
Using matrix identity
\[
\begin{bmatrix}
  \A         & \mathbf{b} \\
  \mathbf{c} & d  
\end{bmatrix}^{-1}
= \begin{bmatrix}
  \A^{-1} + \frac{1}{\id} \A^{-1} \bb \bc \A^{-1} & -\frac{1}{\id} \A^{-1} \bb \\
  -\frac{1}{\id} \bc  \A^{-1} & \frac{1}{\id}
  \end{bmatrix}
\]
where $\id = d - \bc \A^{-1} \bb$,
we obtain for $\G_k$
\[
\G_{k} =  \be_k \M_k^{-1} \by_k
= [\be_{k-1}, 1]
\begin{bmatrix}
  \A^{-1} + \frac{1}{\id} \A^{-1} \bb \bc \A^{-1} & -\frac{1}{\id} \A^{-1} \bb \\
  -\frac{1}{\id} \bc  \A^{-1} & \frac{1}{\id}
\end{bmatrix}
\begin{bmatrix}
  \by_{k-1} \\
  g_k
  \end{bmatrix} ,
\]
and for $\HH_{k}$
\[
\HH_{k} =  \be_k \M_k^{-1} \bz_k
= [\be_{k-1}, 1]
\begin{bmatrix}
  \A^{-1} + \frac{1}{\id} \A^{-1} \bb \bc \A^{-1} & -\frac{1}{\id} \A^{-1} \bb \\
  -\frac{1}{\id} \bc  \A^{-1} & \frac{1}{\id}
\end{bmatrix}
\begin{bmatrix}
  \by_{k-1} \\
  \w g_k
  \end{bmatrix} .
\]
The system of  recurrences of Eq.\eqref{E:G_H_recur}
can be obtained by multiplying out the above two matrix equations. 

\end{proof}

The system of recurrences in Theorem~\ref{Th:rec}
can be solved to obtain the explicit formula for $\G_k$
(as well as $\HH_k$) in terms of $g_i$.
\begin{theorem}[Explicit formula for $\G_{k}$] \label{Th:explicit}
The explicit formula of $\G_k$ and $\HH_k$ of Theorem~\ref{Th:rec} are given by:
\begin{equation} \label{E:explicit}
\begin{aligned}
\G_k &= \dfrac{\displaystyle 1}
             {\displaystyle 1 - \sum_{i=1}^k \sum_{l=0}^{i-1} (\w-1)^l \prod_{j=i-l}^i f_j} - 1 ,\\
 \HH_k &= \dfrac{\displaystyle \sum_{l=0}^k (\w-1)^l \prod_{j=k-l+1}^k f_j}
             {\displaystyle 1 - \sum_{i=1}^k \sum_{l=0}^{i-1} (\w-1)^l \prod_{j=i-l}^i f_j} - 1 
\end{aligned}
\end{equation}
where
\[
 f_i = \frac{g_i}{1+g_i} .
\]
\end{theorem}
\begin{proof}
  The formulas for $\G_k$ and $\HH_k$ can be proved by induction using the
  recurrences in Theorem~\ref{Th:rec}.
Alternatively, Lemma~\ref{L:det_M_k} in Appendix~\ref{A:DM} can be used.
\end{proof}

As a check, when $\w=1$, the only nonvanishing terms in the
denominator of $\G_k$ in Eq.~\eqref{E:explicit} are those when
$l=0$, which leads to
\[
\G_k = \dfrac{\displaystyle 1}
	{\displaystyle 1 - \sum_{i=1}^k f_i} - 1.
\]
This is just the generating function for enumerating the systems with $k$ kinds of
objects, as shown previously~\citep{Kong2006}.

If we don't put any restrictions on the number of integers that can appear in the sequence,
then we use for each integer the following generating function $g_i$:
 \[
 g_i = \sum_{j=1}^{\infty} x_i^j = \frac{x_i}{1-x_i}.
 \]
In this case, $f_i = x_i$,
hence we have
\begin{corollary} \label{C:G}
  When
  \[
    g_i = \sum_{j=1}^{\infty} x_i^j = \frac{x_i}{1-x_i}, \quad 1 \le i \le k,
  \]
the  generating function $\G_k$ is given by  
\begin{equation} \label{E:G_in_x}
 \G_k = \dfrac{\displaystyle 1}
             {\displaystyle 1 - \sum_{i=1}^k \sum_{l=0}^{i-1} (\w-1)^l \prod_{j=i-l}^i x_j} - 1 .
           \end{equation}
\end{corollary}           
For example, when $k=3$ and $f_i=x_i$, we get
\begin{scriptsize}
  \begin{equation} \label{E:G_3}
\G_3 = \frac{1}{1
  - x_1-x_2-x_3
  + x_1 x_2
  + x_2 x_3 
  - x_1 x_2 x_3
  -x_1x_2 \w
  -x_2x_3 \w
  + 2  x_1 x_2 x_3 \w
  -x_1 x_2 x_3 \w^2}
 - 1.
\end{equation}
\end{scriptsize}
Expanding $\G_3$ and taking the coefficient of 
$\bx^\bn = x_1 x_2^2 x_3^2$ gives the generating function of $\s_3(\W; [1,2,2])$:
\begin{equation} \label{E:Johnson2002}
 P_3(w, [1,2,2]) = [x_1 x_2^2 x_3^2] \G_3 =  7+12 \w + 9 \w^2 + 2 \w^3,
\end{equation}
which shows that there are $7, 12, 9$, and $2$ sequences with
$0, 1, 2$, and $3$ \suc{} respectively, out of total $30$ sequences for the specification $[1,2,2]$.
This is the example given by \citet[p.73]{Johnson2002}.
An example for a slightly bigger system (again with $k=3$),
with specification $[2,3,4]$, is given by
\[
 P_3(w, [2,3,4]) = [x_1^2 x_2^3 x_3^4] \G_3 = 93 + 330 \w + 435 \w^2 + 300 \w^3 + 90 \w^4 + 12 \w^5 .
\]

The generating function shown in the form of Eq.~\eqref{E:G_in_x} leads to the following
recurrence relation of the generating functions $P_k(w, \bn)$ and $\s_k(\W; \bn)$
with respect to the number of a particular integer
when the numbers of other integers are fixed.
\begin{theorem}[Recurrence relation for $P_k(w, \bn)$ and $\s_k(\W; \bn)$] \label{Th:P_recur}
  When all the numbers of integers $n_\ell$ are fixed, except for $n_i$ of the $i$th integer,
  the generating functions $P_k(w, \bn)$ satisfy the following recurrence:
  \begin{equation} \label{E:P_recur}
    \sum_{j=0}^d (-1)^j \binom{d}{j} P_k(w, [ n_1, \dots, n_{i}+j, \dots, n_k ]) = 0, \qquad n_i \ge 1,
  \end{equation}
  where the order of recurrence $d = n+1-n_i$.
  
  The number of sequences with
  exactly $\W$ \suc{} $\s_k(\W; \bn)$ also follows a similar recurrence:
  \begin{equation} \label{E:s_recur}
    \sum_{j=0}^d (-1)^j \binom{d}{j} \s_k(\W;  [ n_1, \dots, n_{i}+j, \dots, n_k ]) = 0, \qquad n_i \ge 1.
  \end{equation}
\end{theorem}
\begin{proof}
  Without loosing generality we choose $i=k$.
  From Eq.~\eqref{E:G_in_x} we see that $\G_k$ can be written in a form of
  \begin{equation} \label{E:G_f}
    \G_k = \frac{1}{\fun_1x_1 + \fun_2x_2 + \cdots + \fun_{k-1}x_{k-1}   +1-x_k} - 1,
    \end{equation}
    where $\fun_1$ is a polynomial in $w, x_2, \dots, x_k$,
    $\fun_2$ is a polynomial in $w, x_3, \dots, x_k$, etc.,
    and $\fun_{k-1}$ is a polynomial in $w$ and $x_k$.
    From Eq.~\eqref{E:G_f} we can extract coefficients of $x_1^{n_1}, x_2^{n_2}, \dots, x_{k-1}^{n_{k-1}}$
    stepwisely,
    \begin{align*}
      [x_1^{n_1}] \G_k          =& \frac{(-1)^{n_1} \fun_1^{n_1}} { (\fun_2x_2 + \cdots + \fun_{k-1}x_{k-1}   +1-x_k )^{n_1 + 1}}, \\
      [x_1^{n_1}x_2^{n_2}] \G_k  =& \frac{ Q_2(w, x_3, \dots, x_k) } { (\fun_3x_3 + \cdots + \fun_{k-1}x_{k-1}   +1-x_k )^{n_1 + n_2 + 1}}, \\
                                & \vdots \\
      [x_1^{n_1}x_2^{n_2} \cdots x_{k-1}^{n_{k-1}}] \G_k =& \frac{Q_{k-1}(w, x_k)}{(1-x_k)^{n_1 + n_2 + \cdots + n_{k-1} + 1}},
    \end{align*}
    where $Q_1(w, x_2, , \dots, x_k) = (-1)^{n_1} \fun_1^{n_1}$ is a polynomial function in $w, x_2, , \dots, x_k$,
    $Q_2(w, x_3, , \dots, x_k)$ is a polynomial function in $w, x_3, , \dots, x_k$, and so on,
    and in the last equation $Q_{k-1}(w, x_k)$ is a polynomial in $w$ and $x_k$.
    Since in the denominator of $\G_k$ of Eq.~\eqref{E:G_in_x}
    the degree in each $x_i$ of the monomial terms is $1$,
    the degree of $Q_\ell$ when considered as a polynomial in $x_j$ ($1 \le \ell \le k-1$, $\ell+1 \le j \le k$) 
    is less than the degree of the denominator as a polynomial in $x_j$, which is $n_1+n_2+\cdots+n_\ell+1$.
    In particular, the degree of $Q_{k-1}(w, x_k)$ when considered as a polynomial in $x_k$ is less than $d$.
    From the last equation we obtain the recurrence Eq.~\eqref{E:P_recur}
    \cite[Chapter 4]{stanley_2011}.

    Comparing the coefficients of $w$ of Eq.~\eqref{E:P_recur} we obtain Eq.~\eqref{E:s_recur}.
  \end{proof}
This theorem gives an efficient method to calculate the number of sequences with
  exact $r$ \suc{} for a given specification $\bn$.  
An example of $k=3$ that includes specification $[1,2,2]$ shown above (Eq.~\eqref{E:Johnson2002}) will illustrate this theorem.
If we choose to take the recurrence with respect to the number of the second integer ($n_2$)
and keep $n_1=1$ and $n_3=2$ fixed, we have, for the first
few instances,
\begin{align*}
P_3(w, [1, 1, 2]) &= 5+5w+2w^2 , \\
P_3(w, [1, 2, 2]) &= 7+12w+9w^2+2w^3, \\
P_3(w, [1, 3, 2]) &= 9+21w+21w^2+9w^3, \\
P_3(w, [1, 4, 2]) &= 11+32w+38w^2+24w^3, \\
P_3(w, [1, 5, 2]) &= 13+45w+60w^2+50w^3 .
\end{align*}
These generating functions $P_3(w, \bn)$ satisfy a recurrence of order $d=4$ as specified in Eq.~\eqref{E:P_recur}:
\begin{multline*}
  \binom{4}{0} P_3(w, [1, 1, 2]) - \binom{4}{1} P_3(w, [1, 2, 2]) + \binom{4}{2} P_3(w, [1, 3, 2]) \\
  - \binom{4}{3} P_3(w, [1, 4, 2])  + \binom{4}{4} P_3(w, [1, 5, 2]) = 0.
  \end{multline*}
As a check for $\s_k(\W; [1,2,2])$
with $\W=2$, we extract coefficients of $w^2$ from the above equations and have
$\binom{4}{0} \times 2 - \binom{4}{1} \times 9 + \binom{4}{2} \times 21 - \binom{4}{3} \times 38 + \binom{4}{4} \times 60 = 0$.

\section{Mean and variance} \label{S:mean_var}

As shown in the last section, when there is no 
restriction on the number of integers that can appear in the sequence,
the generating function is given by Eq.~\eqref{E:G_in_x}.
From this generating function the mean and variance of
the number of \suc{} can be obtained.
First we prove a simple lemma.
If we write Eq.~\eqref{E:G_in_x} as
\begin{equation} \label{E:G_in_x_2}
 \G_k = \frac{1}{1-f(\w)} - 1
\end{equation}
with
\[
  f(\w) = \sum_{i=1}^k \sum_{l=0}^{i-1} (\w-1)^l \prod_{j=i-l}^i x_j,
\]
we have

\begin{lemma} \label{L:S}
  \[
   \frac{\pd^m  f(\w) }{\pd\w^m} \Big |_{\w=1} = m! S_{m}, \quad 0 \le m \le k-1 ,
   \]
   where
   \[
    S_{m} = \sum_{i=1}^{k-m}  x_{i} x_{i+1} \cdots x_{i+m}  .
   \]
\end{lemma}    
\begin{proof}
  The term involving $\w-1$ of $\frac{\pd^m  f(\w) }{\pd\w^m}$ is of the form
  $l^{(m)} (\w-1)^{l-m}$.
  This term does not vanish at $\w=1$ only when $l=m$.
  Hence
  \[
    \frac{\pd^m  f(\w) }{\pd\w^m} \Big |_{\w=1} = m! \sum_{i=m+1}^k \prod_{j=i-m}^i x_j = m! S_m.
  \]
\end{proof}
  
\begin{theorem}[Mean of \suc{}]  \label{Th:mean}
  When there is no restriction on the number of integers in the sequences,
  the mean of \suc{} for a specification of $\bn=[n_1, n_2, \dots, n_k]$ is given by
  \begin{equation} \label{E:mean}
    \mu_k(\bn) = \frac{n_1n_2 + n_2n_3 + \cdots + n_{k-1} n_k}{n}
    \end{equation}
    where $n=\sum_{i=1}^k n_i$.
\end{theorem}
\begin{proof}
  By taking the derivative of $\G_k$ in Eq.~\eqref{E:G_in_x} or Eq.~\eqref{E:G_in_x_2} with respect to $\w$
  and then letting $\w=1$, and using Lemma~\ref{L:S},
  we obtain
  \begin{equation} \label{E:d1}
  \frac{ \pd \G_k} {\pd \w} \Big |_{\w=1} =
  \frac{\sum_{i=1}^{k-1} x_i x_{i+1}}{(1 - \sum_{i=1}^{k} x_i)^2} = \frac{S_1}{(1-S_0)^2}.
  \end{equation}
  Expanding the last equation, we can extract the coefficients of $\bx^{\bn}$ as
  \begin{multline*}
    [\bx^{\bn} ] \frac{ \pd \G_k} {\pd \w} \Big |_{\w=1}
    = (n-1)!
    \left[
        \frac{1}{(n_1-1)! (n_2 - 1)! n_3! \cdots n_k!}
      + \frac{1}{n_1! (n_2 - 1)! (n_3-1)! \cdots n_k!}   \right.\\
     \left. + \cdots
      + \frac{1}{n_1!n_2! \cdots (n_{k-1} -1 )! (n_{k} -1 )!}
      \right] .
  \end{multline*}
  Dividing the above equation
  by the total number of sequences $n!/(n_1! n_2! \cdots n_k!)$, we obtain
  the mean of \suc{} in Eq.\eqref{E:mean}.
\end{proof}

Similarly, the variance can be calculated by first
taking the second order derivative
of $\G_k$ and setting $\w=1$ to  get the second factorial moment $E[\WW(\WW-1)]$,
and then use the relation
\[
 \var_k^2(\bn) = E[\WW(\WW-1)] - \mu_k(\bn)^2 + \mu_k(\bn).
\]
The proof is similar to
that of Theorem~\ref{Th:mean} but with more complicated manipulations.
In the end, a compact formula for the variance can be obtained.

\begin{theorem}[The second factorial moment and variance] \label{Th:e12_var} 
  When there is no restriction on the number of integers in the sequences,
  the second factorial moment and variance of \suc{} for a specification of $\bn=[n_1, n_2, \dots, n_k]$
  are given by
  \begin{equation}  \label{E:e12}
  E[\WW(\WW-1)] = 
  \frac{1}{n(n-1)} 
  \left[
    T_1^2 - \sum_{i=1}^{k-1} n_i n_{i+1} (n_{i} + n_{i+1} - 1) 
    \right] ,
  \end{equation}
  and
  \begin{equation}  \label{E:var}
    \var_k^2(\bn) = \frac{1}{n^2(n-1)}  T_1^2
    + \frac{1}{n-1} T_1
    - \frac{1}{n(n-1)}  \sum_{i=1}^{k-1} n_i n_{i+1} (n_{i} + n_{i+1} ) 
  \end{equation}
  where
  \[
   T_1 = \sum_{i=1}^{k-1} n_i n_{i+1} .
  \]
\end{theorem}
\begin{proof}
  By taking the second order derivative of $\G_k$
  in Eq.~\eqref{E:G_in_x_2} with respect to $\w$ we obtain
  \[
    \frac{ \pd^2 \G_k} {\pd \w^2} =  \frac{2 (f')^2}{(1-f)^3} + \frac{f''}{(1-f)^2} .
    \]
    Using Lemma~\ref{L:S}, we have
    \begin{equation} \label{E:d2}
     \frac{ \pd^2 \G_k} {\pd \w^2} \Big |_{\w=1} =  \frac{2 S_1^2}{(1-S_0)^3} + \frac{2 S_2}{(1-S_0)^2} .
     \end{equation}
    After expanding the equation above,
    extracting the coefficient of $\bx^{\bn}$,
    dividing by the
    total number of sequences $n!/(n_1! n_2! \cdots n_k!)$,
    and simplifying the results,  we arrive at Eq.~\eqref{E:e12} and Eq.~\eqref{E:var}.
\end{proof}

A few explicit examples are given here for checking the formula.  
For $k=3$ and $\bn=[1, 2, 2]$, the generating function $P_3(w, [1,2,2])$ is given in Eq.~\eqref{E:Johnson2002}.
To calculate the mean and variance explicitly from the generating function,
the total number of sequences is $7 + 12 + 9 + 2 = 30$, and
the mean is $(7 \times 0 + 12 \times 1 + 9 \times 2 + 2 \times 3) / 30 = 6/5$,
while the variance is given by
$(7 \times 0^2 + 12 \times 1^2 + 9 \times 2^2 + 2 \times 3^2) / 30 - (6/5)^2 = 19/25$.
These results can check with Eqs.~\eqref{E:mean} and \eqref{E:var}, respectively.

For $k=5$ and $\bn=[2, 3, 1, 2, 2]$, the generating function for $\s_5(\W; \bn)$ is
\begin{multline*}
  P_5(w, [2, 3, 1, 2, 2]) = [x_1^2 x_2^3 x_3 x_4^2 x_5^2] \G_5 \\
  = 14346 + 26394 \w + 21480\w^2 + 10020\w^3 + 2850\w^4 + 474 \w^5 + 36\w^6,
\end{multline*}
and the mean, second factorial moment, and variance for this specification are
$  \mu_5 = \frac{3}{2} $, 
$  E[\WW(\WW-1)] = \frac{88}{45} $,
and
$\var_5^2 = \frac{217}{180} $.

As a special case for Theorem~\ref{Th:mean} and Theorem~\ref{Th:e12_var},
for specification $\bn=[\bone]$,
the mean and variance  have simple forms.
\begin{corollary} \label{C:1-mean-var}
  When $\bn=[\bone]$, the mean,  the second factorial moment, and variance are
\begin{align*}
  \mu_k(\bone) &= \frac{k-1}{k} = 1 - \frac{1}{k} ,\\
   E[\WW(\WW-1)]   &=  \frac{k-2}{k}, \\
 \var_k^2(\bone) &= \frac{k^2-k-1}{k^2} .
\end{align*}
\end{corollary}

\begin{remark}
By using Fa\`{a} di Bruno's formula on Eq.~\eqref{E:G_in_x_2} as $h(f(\w))$ with $h=1/(1-f)$, we have the expression
of the $m$th derivative of the generating function $\G_k$:
  \begin{align*}
  \frac{\pd^m \G_k}{\pd \w^m} &=  \sum_{m_1, m_2, \dots, m_m} \frac{m!}{m_1! m_2! \dots, m_m !}
  h^{(m_1 + m_2 + \dots + m_m)} \prod_{j=1}^m \left( \frac{ f^{(j)} (\w) } { j! } \right )^{m_j}   \\
  &= \sum_{m_1, m_2, \dots, m_m}   \multinom{m_1 + \dots + m_m}{m_1, m_2, \dots, m_m}
  \frac{m! }{(1-f(\w))^{m_1  + \dots + m_m + 1}}
  \left( \frac{ f^{(j)} (\w) } { j! } \right )^{m_j}  
  \end{align*}
  where the sum is over all $m$-tuples of nonnegative integers $(m_1, m_2, \dots, m_m)$ satisfying the constraint
  \[
   m_1 + 2 m_2 \cdots + m m_m =  m .
   \]
Using Lemma~\ref{L:S} we obtain at $\w=1$:
   \begin{equation}  \label{E:dm}
    \frac{ \pd^m \G_k} {\pd \w^m} \Big |_{\w=1} =
    m! \sum_{m_1, m_2, \dots, m_m}   \multinom{m_1 + m_2 + \dots + m_m}{m_1, m_2, \dots, m_m}
    \frac{ \prod_{j=1}^m S_j^{m_j} }{(1-S_0)^{m_1 + m_2 + \dots + m_m + 1}} .
   \end{equation}
Eq.~\eqref{E:d1} and Eq.~\eqref{E:d2}
are special cases of Eq.~\eqref{E:dm} with $m=1$ and $m=2$
respectively.
With Eq.~\eqref{E:dm}
we can obtain higher order factorial moments
for \suc{} in random permutations of arbitrary multisets.

\end{remark}

\begin{appendices}

\section{Determinant of $\M_k$} \label{A:DM}
From Eqs.~\eqref{E:matrix} and ~\eqref{E:H} we see that the denominators of
$\G_k$ and $\HH_k$ are determined by the determinant of matrix $\M_k$.
In this Appendix we give an explicit formula for the determinant of $\M_k$.

\begin{lemma}[Determinant of $\M_k$] \label{L:det_M_k}
  The determinant of $\M_k$ is given by
\begin{equation} \label{E:det_M_k_exp}
 \det(\M_k) 
 = \prod_{i=1}^k (1+g_i)
 \left[
   1 - \sum_{i=1}^k \sum_{l=0}^{i-1} (\w-1)^l \prod_{j=i-l}^i f_j 
   \right] 
\end{equation}
with
$
 f_i = \frac{g_i}{1+g_i}.
$
\end{lemma}

\begin{proof}

From the definition of $\M_k$ in Eq.~\eqref{E:M} we have the determinant of $\M_k$ as:
\begin{align} \label{E:det_M_k}
\det(\M_k)
&= \begin{vmatrix}
  1 & \mathbf{0} \\
  \by_k & \M_k
\end{vmatrix}
=
\begin{vmatrix}
  1 & 0 & 0 & 0 & \cdots & 0\\
  g_1 &         1 & -g_1 \w  & -g_1   & \cdots & -g_1 \\
  g_2 & -g_2  &  1           & -g_2 \w  & \cdots & -g_2  \\
      \vdots & \vdots       & \vdots       & \vdots & \vdots  & \vdots     \\
  g_{k-1}  &-g_{k-1}  & \cdots & -g_{k-1}   & 1        & -g_{k-1} \w \\
  g_k     & -g_k     & \cdots & -g_k      & -g_k  & 1 
\end{vmatrix} \nonumber \\
&=
\begin{vmatrix}
  1 & 1 & 1 & 1 & \cdots & 1\\
  g_1 &       1+g_1 & (1-\w) g_1   & 0   & \cdots & 0 \\
  g_2 &  0  &  1+g_2           & (1-\w) g_2   & \cdots & 0  \\
      \vdots & \vdots       & \vdots       & \vdots & \vdots  & \vdots    \\
  g_{k-1}  & 0  & \cdots & 0   & 1+g_{k-1}        & (1-\w) g_{k-1}  \\
  g_k     & 0  & \cdots & 0      & 0  & 1+ g_k
\end{vmatrix}
= \det( \U_{k} ) \nonumber \\
&= (1+ g_k) \det(\U_{k-1}) + (-1)^k g_k \det(\V_{k-1}), 
\end{align}
where in the first step we use Laplace expansion to
extend the matrix $\M_k$ to a $(k+1) \times (k+1)$ matrix
by attaching the first row and the first column,
in the second step the first column is added to the other columns,
and the last step is the Laplace expansion on the last row
of the matrix, with matrices $\U_k$ and $\V_k$ defined below:
\[
\U_k =
 \begin{bmatrix}
  1 & 1 & 1 & 1 & \cdots & 1\\
  g_1 &       1+g_1 & (1-\w) g_1   & 0   & \cdots & 0 \\
  g_2 &  0  &  1+g_2           & (1-\w) g_2   & \cdots & 0  \\
      \vdots & \vdots       & \vdots       & \vdots & \vdots   & \vdots   \\
  g_{k-1}  & 0  & \cdots & 0   & 1+g_{k-1}        & (1-\w) g_{k-1}  \\
  g_k     & 0  & \cdots & 0      & 0  & 1+ g_k
\end{bmatrix},
\]
and 
\[
\V_k =
 \begin{bmatrix}
  1 & 1 & 1 & \cdots & 1\\
         1+g_1 & (1-\w) g_1   & 0   & \cdots & 0 \\
    0  &  1+g_2           & (1-\w) g_2   & \cdots & 0  \\
      \vdots & \vdots       & \vdots       & \vdots & \vdots      \\
   0  & \cdots    & 1+g_{k-1}        & (1-\w) g_{k-1} & 0 \\
   0  & \cdots & 0   & 1+g_{k}        & (1-\w) g_{k}
\end{bmatrix} .
 \]
 
By expansion on the last column of matrix $\V_k$,
 the recurrence for determinant of $\V_k$
 can be obtained as 
\[
\det(\V_k) = (1-\w) g_k \det(\V_{k-1})
 + (-1)^k \prod_{i=1}^k (1+g_i)
\]
with $\det(\V_0) = 1$.
If we define
\[
 f_i = \frac{g_i}{1+g_i},
\]
$\det(\V_k)$ can be solved from the recurrence as 
\[
 \det(\V_k) = (-1)^k \prod_{i=1}^k (1+g_i)
\sum_{i=0}^k (\w-1)^i \prod_{j=0}^{i-1} f_{k-j} .
 \]

The determinant of $\M_k$ can be solved from Eq.~\eqref{E:det_M_k} as 
\[
\det(\M_k) = \det(\U_k)
= \prod_{i=1}^k (1+g_i)
+ (-1)^k \sum_{i=0}^{k-1} (-1)^i g_{k-i} \det(\V_{k-i-1}) \prod_{j=k-i+1}^{k}
   (1+g_{j})  .
\]
Substituting the expression of $\det(\V_k)$
with further simplification gives 
Eq.~\eqref{E:det_M_k_exp}.

\end{proof}

\section{Explicit formulas when $k=2$ and $k=3$} \label{A:k2k3}
For a general $k$ and an arbitrary specification $\bn$,
it's not easy to write down an explicit expression for $\Q_k(\w; \bn)$ or $\s_k(\W; \bn)$.
As $k$ increases, the formulas become more complicated.
For small values of $k$, however, the expressions are manageable.
In this appendix we give expressions for $k=2$ and $k=3$.

\subsection{$k=2$}
For $k=2$, from Eq.~\eqref{E:G_in_x} we get
\[
 \G_2 = \frac{1}{1-(x_1 + x_2 + x_1 x_2 (\w-1) )} - 1.
 \]
Using the binomial theorem to expand $\G_2$ we obtain 
 \[
  \Q_2(\w; [n_1, n_2]) = \sum_{i=0}^{n_2} \binom{n_1 + n_2 - i}{n_2} \binom{n_2}{i} (\w-1)^i
  \]
Using the identity
\[
\sum_i \binom{s_1+s_2-i}{s_2} \binom{s_2}{i} \binom{i}{j} (-1)^i
= (-1)^j \binom{s_1}{j}\binom{s_2}{j},
\]
we have
\begin{corollary}
  When $k=2$, $\Q_2(\w; [n_1, n_2])$ is
\[
 \Q_2(\w; [n_1, n_2]) = \sum_{\W=0}^{\min(n_1, n_2)} \binom{n_1}{\W} \binom{n_2}{\W} \w^{\W} .
 \]
\end{corollary}
From the Corollary we obtain the number of sequences with \suc{} as
\[
  \s_k(\W; [n_1, n_2]) = [\w^{\W}] \Q_2(\w; [n_1, n_2]) = \binom{n_1}{\W} \binom{n_2}{\W} .
\]

\subsection{$k=3$}
For $k=3$,
by expanding Eq.~\eqref{E:G_3} we get the generating function $\Q_3(\w; [n_1, n_2, n_3])$:

\begin{corollary}
  When $k=3$, $\Q_2(\w; [n_1, n_2, n_3])$ is
 \[
 \Q_3(\w; [n_1, n_2, n_3]) = \sum_{i=0}^{n_1} \sum_{j=0}^{n_2} \sum_{l=0}^{n_3}
 \multinom{n_1+n_2+n_3 - i-j-2l}{i, j , l, n_1-i-l, n_2 -i-j-l, n_3-j-l}
 (\w-1)^{i+j+2l}  
 \]
 with $n_1-i-l \ge 0$,
 $n_2-i-j-l \ge 0$,
 and $n_3-j-l \ge 0$ in the sums.
\end{corollary}

\end{appendices}

\bibliographystyle{elsarticle-harv}
\bibliography{run}

\end{document}